\documentclass[12pt]{article}  

\usepackage{amsmath,amsthm,amssymb,mathdots,hyperref,shuffle,array,enumerate,float}
\usepackage{listings}
\usepackage{graphicx}
\usepackage{setspace}
\usepackage[vcentermath]{youngtab}
\usepackage{color}
\allowdisplaybreaks

\usepackage{ifthen}

\newtheorem{theorem}{Theorem}[section]
\newtheorem{proposition}[theorem]{Proposition}
\newtheorem{definition}[theorem]{Definition}
\newtheorem{lemma}[theorem]{Lemma}
\newtheorem{corollary}[theorem]{Corollary}
\newtheorem{conjecture}[theorem]{Conjecture}

\newtheorem{remark}[theorem]{Remark}
\numberwithin{equation}{section}

\newcommand{\Z}{{\mathbb Z}}
\newcommand{\N}{{\mathbb N}}

\newcommand{\Q}{{\mathbb Q}}
\newcommand{\MZ}{\mathcal Z}
\DeclareMathOperator{\calA}{\mathcal{A}}
\DeclareMathOperator{\MD}{\mathcal{MD}}
\DeclareMathOperator{\bMD}{\mathcal{BD}}

\newcommand{\zq}[1][]{%
	\ifthenelse{ \equal{#1}{} }
	{\ensuremath{\mathcal{Z}_{q}}}
	{\ensuremath{\mathcal{Z}_{q,#1}}}
}  

\newcommand{\zqc}[1][]{%
	\ifthenelse{ \equal{#1}{} }
	{\ensuremath{\mathcal{Z}^\circ_{q}}}
	{\ensuremath{\mathcal{Z}^\circ_{q,#1}}}
}

\DeclareRobustCommand{\mb}{\genfrac{[}{]}{0pt}{}}

\DeclareRobustCommand{\mt}{\genfrac{|}{|}{0pt}{}}

\newcommand{\num}{ \operatorname{num} }

\DeclareMathOperator{\odd}{\mathsf{O}}
 \DeclareMathOperator{\ev}{\mathsf{E}}
\DeclareMathOperator{\tM}{\widetilde{\mathsf{M}}}
\DeclareMathOperator{\Dx}{\mathsf{D}}
\DeclareMathOperator{\Sx}{\mathsf{S}}
\DeclareMathOperator{\OZ}{\mathsf{Z}}
\DeclareMathOperator{\sera}{\mathsf{a}}

\newcommand{\gr}{ \operatorname{gr}}
\newcommand{\filwle}{\operatorname{Fil}^{\operatorname{W},\operatorname{D}}}
\newcommand{\filw}{ \operatorname{Fil}^{\operatorname{W}}}
\newcommand{\fild}{ \operatorname{Fil}^{\operatorname{D}}}
\newcommand{\grwl}{ \operatorname{gr}^{\operatorname{W},\operatorname{D}}}
\newcommand{\grw}{ \operatorname{gr}^{\operatorname{W} }}
\newcommand{\grl}{ \operatorname{gr}^{\operatorname{D} }}

\title{{ \bf  A dimension conjecture for $q$-analogues of multiple zeta values}}
\author{ Henrik Bachmann,  Ulf K\"uhn}

\date{\today}

\begin{document}
	
	\maketitle
	\begin{abstract} 
We study a class of $q$-analogues of multiple zeta values given by certain formal $q$-series with rational coefficients. After introducing a notion of weight and depth for these  $q$-analogues of multiple zeta values we present dimension conjectures for the spaces of their weight- and depth-graded parts, which have a similar shape as the conjectures of Zagier and Broadhurst-Kreimer for multiple zeta values.
	\end{abstract}
	
	\section{Introduction}
Multiple zeta values are real numbers appearing in various areas of mathematics and theoretical physics. By a $q$-analogue of these numbers one usually understand $q$-series, which degenerate to multiple zeta values as $q\rightarrow 1$. The algebraic structure of several models of $q$-analogues has been the subject of recent research (see \cite{Zh} for an overview). Besides a conjecture of Okounkov in \cite{O} for the dimension of the weight-graded spaces 
for a specific such model, no conjectures for the dimensions of the spaces of any of these $q$-analogues in a given weight and depth have occurred in the literature. The purpose of this work is to introduce a space of $q$-series which contains a lot of these models and to present conjectures on the dimensions of their weight- and depth-graded parts.	
	For natural numbers $s_1\geq2$, $s_2,\dots,s_l \geq 1$ define the multiple zeta value (MZV)
	\[ \zeta(s_1, \dots,s_l) = \sum_{n_1 >  \dots > n_l > 0} \frac{1}{n_1^{s_1} \dots n_l^{s_l}}\,. \]
	By $s_1+ \dots+s_l$ we denote its weight, by $l$ its depth and we write $\MZ$ for the $\Q$-vector space spanned by all MZVs. It is a well-known fact that the space $\MZ$ is a $\Q$-algebra and that there are two different ways, known as respectively the stuffle and shuffle product formulas, which express the product of two MZVs as a $\Q$-linear combination of MZVs. These two ways of writing the product gives a large family of $\Q$-linear relations between MZVs in a fixed weight, known as double shuffle relations. 
Conjecturally all relations between MZVs follow from these type of relations. 
In particular it is conjectured that the algebra $\MZ$ is graded by the weight.
Let $\MZ_k$ denote  the $\Q$-vector space  spanned by the MZVs of weight $k$, then there is the following famous dimension conjecture due to Zagier:


	\begin{conjecture}[Zagier \cite{Za}]\label{conj:zagier}
		The following identity holds: 
		\begin{align*}
		\sum_{k \geq 0} \dim_\Q \left(
		\MZ_k
		\right) x^k    = \frac{1}{1-x^2-x^3}.
		\end{align*}
	\end{conjecture}
	A stronger version of this conjecture was later proposed by Hoffman (\cite{Hoffman}), which states that the $\zeta(s_1,\dots,s_l)$ with $s_j \in \{2,3\}$ form a basis of $\MZ$. So far it is only known, due to a result of Brown (\cite{Brown}), that these MZVs span the space $\MZ$.
	Conjecture \ref{conj:zagier} has a refinement by Broadhurst and Kreimer who proposed the following conjecture on the dimension of the weight- and depth-graded parts:
	\begin{conjecture}[Broadhurst-Kreimer \cite{brokrei}] \label{conj:brokrei}
		The generating series of the dimensions of the weight- and depth-graded parts of multiple zeta values is given by
		\begin{align*}
		\sum_{k,l \geq 0} \dim_\Q \left(\grl_{l} \MZ_k\right) x^k y^l  &
		= 
		\frac{1 +\ev_2(x) y}{ 1- \odd_3(x) y + \Sx(x) y^2 - \Sx(x) y^4},
		\end{align*}
		where 
		\begin{align*}
		\ev_2(x) = \frac{ x^2}{1-x^2},  \quad
		\odd_3(x) = \frac{ x^3}{1-x^2} , \quad
		\Sx(x) = \frac{x^{12}}{(1-x^4)(1-x^6)}.
		\end{align*}
	\end{conjecture} 
Observe that $\ev_2(x)$ (resp. $\odd_3(x)$) is the generating series of the number of even (resp. odd) zeta values
and $\Sx(x)$ is the generating series for the dimensions of cusp forms for $\operatorname{SL}_2(\Z)$.
Furthermore, by setting $y=1$ on the right-hand side of the Broadhurst-Kreimer conjecture one obtains precisely the right-hand side in the Zagier conjecture.	
	We are interested in conjectures  similar to  the above in the context of  $q$-analogues of multiple zeta values. 
	There are various different models of $q$-analogues for multiple zeta values. For most of these models the algebraic setup, i.e. analogues of the stuffle and the shuffle product, is well understood (See for example \cite{Ba5}, \cite{ems}, \cite{S}, \cite{yt}, \cite{Zh}). 
	The problem of understanding the dimension of the weight-graded spaces has been considered in \cite{Ba5}, \cite{BK}, \cite{ems}, \cite{O}, \cite{S}, \cite{yt}, \cite{Zh}. On the other hand possible analogues of the Broadhurst-Kreimer conjecture for these $q$-analogues have not been proposed yet.
 
	Now we will define the $q$-analogues of multiple zeta values we consider in this paper. For $s_1,\dots,s_l \geq 1$ and polynomials $Q_1(t) \in t \Q[t]$ and $Q_2(t)\dots,Q_l(t) \in \Q[t]$ we define
	\begin{equation} \label{def:Z1}
	\zeta_q(s_1,\dots,s_l ; Q_1 , \dots , Q_l) = \sum_{n_1>\cdots>n_l>0} \frac{Q_1(q^{n_1}) \dots Q_l(q^{ n_l})}{(1-q^{n_1})^{s_1}\cdots (1-q^{n_l})^{s_l}} \,.
	\end{equation}
	This series can be seen as a $q$-analogue\footnote{These type of series are often called \emph{modified} $q$-analogues of multiple zeta values, since one needs to multiply by $(1-q)^{s_1+\dots+s_l}$ before taking the limit $q\rightarrow 1$} of $\zeta(s_1,\dots,s_l)$, since we have for $s_1 > 1$  \[ \lim\limits_{q\rightarrow 1} (1-q)^{s_1+\dots+s_l} \zeta_q(s_1,\ldots,s_l; Q_1 , \dots , Q_l) = Q_1(1) \dots Q_l(1)\cdot \zeta(s_1,\ldots,s_l) \,. \]
	We only consider the case where $\deg(Q_j) \leq s_j$ and consider the following $\Q$-algebra: 
	\[ \zq := \Big \langle   \zeta_q(s_1,\dots,s_l ; Q_1 , \dots , Q_l) \,\,\big|\,\, l\ge 0 ,\,s_1,\dots,s_l\geq1 ,\,\deg(Q_j) \leq s_j 
	\Big\rangle_\Q\,. \]
	Contrary to the case of MZVs, the number $s_1+\dots+s_l$ does not give a good notion of weight for the $\zeta_q$, since for example $\zeta_q(s;Q) = \zeta_q(s+1,Q \cdot (1-t))$. Also the number $l$ will not be used to define the depth. Instead we will consider a class of $q$-series which also span the space $\zq$ and use these series to define a weight and a depth filtration on $\zq$.
	For $s_1,\dots,s_l \geq 1$, $r_1,\dots,r_l \geq 0$ these $q$-series are given by
	\begin{align}\label{eq:defbibracketintro}
	\mb{s_1, \dots , s_l}{r_1,\dots,r_l} :=\sum_{\substack{u_1 > \dots > u_l> 0 \\ v_1, \dots , v_l >0}} \frac{u_1^{r_1}}{r_1!} \dots \frac{u_l^{r_l}}{r_l!} \cdot \frac{v_1^{s_1-1} \dots v_l^{s_l-1}}{(s_1-1)!\dots(s_l-1)!}   \cdot q^{u_1 v_1 + \dots + u_l v_l} \in \Q[[q]]\,.
	\end{align}
	We refer to these  $q$-series as \emph{bi-brackets} of depth $l$ and 
	weight $s_1+\dots+s_l+r_1+\dots+r_l$. 
	They were introduced by the first author in \cite{Ba2} and their algebraic structure is well-understood and described in the papers \cite{Ba2}, \cite{Ba4}, \cite{Ba5}, \cite{Z}. 
	The bi-brackets  have a natural connection to quasi-modular forms (for $\operatorname{SL}_2(\Z)$), since for even $k$ the Fourier expansion of the classical Eisenstein series $G_k$ of weight $k$ is given by $\mb{k}{0}$ plus an appropriate constant term. 
	In particular the space of quasi-modular forms with rational coefficients, which is given by $\Q[G_2,G_4,G_6]$, is a sub-algebra of the space $\zq$. 
	
	As we will see in Theorem \ref{thm:bdasqmzv} the bi-brackets span the space $\zq$ and therefore we can define a weight and a depth filtration by using the notion of weight and depth of bi-brackets. We point to the fact that $\zq$ is not graded by the weight, i.e. 
 the weight graded spaces
  $\gr_{k}^{W} \zq$  are in general not isomorphic to  the $\Q$-vector spaces spanned by bi-brackets of weight $k$. 
	In analogy to the Zagier and Broadhurst-Kreimer conjecture we conjecture the following.
	\begin{conjecture}  \phantomsection\label{conj:bkforbibra}  
 	\begin{enumerate}[(i)]   
			\item The dimension of the weight graded parts of $\zq$ is given by
			\[
			\sum_{k \geq 0} \dim_\Q \left(\gr_{k}^{W} \zq \right)  x^k  = \frac{1}{ 1 -x -x^2-x^3 + x^6 +x^7 +x^8 +x^9 }.
			\]
			\item The dimension of the weight and depth graded parts of $\zq$ is given by 	 	\begin{align*}
			\sum_{k,l \geq 0} \dim_\Q \left( \grwl_{k,l} \zq \, \right) x^k y^l  =\frac{
			 1+ \Dx(x) \ev_2(x) y + \Dx(x)\Sx(x) y^2
			}{	1 - \sera_1(x)\, y  +\sera_2(x)\, y^2 - \sera_3(x) \, y^3 -\sera_4(x)\, y^4 +\sera_5(x)\, y^5  }   \,  
			,
			\end{align*}	
			where   $\Dx(x)=1/(1-x^2)$, $\odd_1(x) = x/(1-x^2)$ and   $\ev_2(x), \Sx(x)$ are as in Conjecture \ref{conj:brokrei} and
			\begin{align*}
			\sera_1(x) &=  \Dx(x) \,\odd_1(x)\,, \qquad &&\sera_2(x) =  \Dx(x) \, \sum_{k\geq 1} \dim_\Q ( M_{k} (  \operatorname{SL}_2(\Z))^2  \, x^k  \,,  \\
			\sera_3(x) &=\sera_5(x)=  \odd_1(x)\, \Sx(x)
			\,,\qquad
			&&\sera_4(x) =  \Dx(x) \, \sum_{k\geq 1} \dim_\Q ( S_{k} (  \operatorname{SL}_2(\Z))^2  \, x^k \,.
			\end{align*} 
			Here $M_{k} (  \operatorname{SL}_2(\Z))$ and $S_{k} (  \operatorname{SL}_2(\Z))$ denote the spaces of modular forms and cusp forms for $\operatorname{SL}_2(\Z)$ of weight $k$. 
		\end{enumerate} 
	\end{conjecture}
	Note that setting $y=1$ in (ii) implies (i), this holds because of the  formula 
	\[
	\sum_{k\geq 0} \dim_\Q ( M_{k} (  \operatorname{SL}_2(\Z))^2  \, x^k = 
	\frac{1+x^{12}}{(1-x^4)(1-x^6)(1-x^{12})},
	\]
	which is straightforward to prove. 
	
	In the Broadhurst-Kreimer conjecture the numerator $1 +\ev_2(x) y$ can be interpreted as the generating series of $\dim_\Q \gr_{k,l}^{W,D} \Q[\zeta(2)]$,
	i.e.
	\[ \sum_{k,l \geq 0} \dim_\Q \grwl_{k,l}\left(\Q[\zeta(2)]\right) \, x^k y^l  = 1 +\ev_2(x) y\,. \]
As we will see in Proposition \ref{prop:qmf} the numerator in Conjecture \ref{conj:bkforbibra} (ii)  is essentially the generating series for the weight- and depth-graded dimensions of the quasi-modular forms, since 
$\Dx(x) \ev_2(x)$
counts the number of Eisenstein series and their derivatives and 
$\Dx(x) \Sx(x)$
corresponds to the number of cusp forms and their derivatives. Therefore it is reasonable to expect that
		\[ \sum_{k,l \geq 0} \dim_\Q \grwl_{k,l}\left(\Q[G_2,G_4,G_6]\right) \, x^k y^l  \overset{?}{=} 
		 1+ \Dx(x) \ev_2(x) y + \Dx(x)\Sx(x) y^2
		\,. \]		
		In some sense the quasi-modular forms in the context of $q$-analogues of multiple zeta values play the role of the even single zeta values (see also \cite{Go}, \cite{Zu3}). 
		
For $k \le 15$ we determined, by calculating a large number of coefficients, lower bounds for  $\dim_\Q \left(\gr_{k}^{W} \zq \right)$, which equal the expected dimensions in Conjecture \ref{conj:bkforbibra} (i). 
Furthermore, see Theorem \ref{thm:upperBK} below, Conjecture \ref{conj:bkforbibra} (i) 
actually holds for $k \le 7$. For the refined Conjecture \ref{conj:bkforbibra} (ii) 
our computer experiments provide us with lower bounds, which again equal the expected dimensions, in the range given by Table \ref{table:exprange} on page \pageref{table:exprange}.\\

{\bf Acknowledgments}\\
We would like to thank N. Matthes for the careful  reading of our manuscript and his valuable comments. The first author would also like to thank the Max-Planck Institute for Mathematics in Bonn for hospitality and support.

	\section{$q$-analogues of MZVs and bi-brackets}

	Usually a function $f(q)$  is called a $q$-analogue of multiple zeta value, if $\lim\limits_{q \rightarrow 1} f(q)$ is a multiple zeta value. There are various different models of $q$-analogues in the literature (See \cite{Zh} for a nice overview). One of the first models was studied by Bradley \cite{db} and Zhao \cite{Zh1} independently. This model is given for $s_1\geq 2, s_2,\dots s_l \geq 1$ by the $q$-series 
	\begin{equation} \label{eq:qana} \sum_{n_1>\cdots>n_l>0} \frac{q^{(s_1-1) n_1} \dots q^{(s_l-1) n_l}}{\{n_1\}_q^{s_1}\cdots \{n_l\}_q^{s_l}} \,,
	\end{equation}
	with $\{n \}_q= \frac{1-q^n}{1-q}$ being the usual $q$-integer. Taking the limit $q \rightarrow 1$ in above sum one obtains $\zeta(s_1,\dots,s_l)$. For a cleaner description of the algebraic structure and (in our case) a connection to modular forms it is convenient to consider a modified version of \eqref{eq:qana} by removing the factor $(1-q)^{s_1+\dots+s_l}$, i.e. to consider the series
	
	\begin{equation}\label{eq:defzetaq} \zeta^{\text{BZ}}_q(s_1,\dots,s_l) =   \sum_{n_1>\cdots>n_l>0} \frac{q^{(s_1-1) n_1} \dots q^{(s_l-1) n_l}}{(1-q^{n_1})^{s_1}\cdots (1-q^{n_l})^{s_l}} \,,
	\end{equation}
	which then satisfies $\lim\limits_{q\rightarrow 1} (1-q)^{s_1+\dots+s_l} \zeta_q(s_1,\ldots,s_l) = \zeta(s_1,\ldots,s_l)$.

	In a greater generality we will consider for $s_1,\dots,s_l \geq 1$ and polynomials $Q_1(t) \in t \Q[t]$ and $Q_2(t)\dots,Q_l(t) \in \Q[t]$ sums of the form 
	\begin{equation} \label{def:Z}
	\zeta_q(s_1,\dots,s_l ; Q_1 , \dots , Q_l) = \sum_{n_1>\cdots>n_l>0} \frac{Q_1(q^{n_1}) \dots Q_l(q^{ n_l})}{(1-q^{n_1})^{s_1}\cdots (1-q^{n_l})^{s_l}} \,.  
	\end{equation}
	The condition $Q_1(t) \in t \Q[t]$ ensures that this is an element in $\Q[[q]]$. In contrast to \eqref{eq:defzetaq} we also allow $s_1=1$ in our setup, i.e. we also include  $q$-analogues of the non-convergent multiple zeta values. 
	In the case $s_1 > 1$ we can (by the same arguments as in \cite{BK} Proposition 6.4) again take the limit $q \rightarrow 1$ after multiplying by $(1-q)^{s_1+\dots+s_l}$ , which gives
	\[ \lim\limits_{q\rightarrow 1} (1-q)^{s_1+\dots+s_l} \zeta_q(s_1,\ldots,s_l; Q_1 , \dots , Q_l) = Q_1(1) \dots Q_l(1)\cdot \zeta(s_1,\ldots,s_l) \,. \]
	Almost all models of $q$-analogues in the literature are given by sums of the form \eqref{def:Z}.  In the following we always set $\zeta_q(s_1,\ldots,s_l; Q_1 , \dots , Q_l)=1$ for the case $l=0$.
	We will consider the following spaces spanned by the series \eqref{def:Z} of a particular kind
	\[ \zq = \Big \langle   \zeta_q(s_1,\dots,s_l ; Q_1 , \dots , Q_l) \,\,\big|\,\, l\ge 0 ,\,s_1,\dots,s_l\geq1 ,\,\deg(Q_j) \leq s_j 
	\Big\rangle_\Q\,,  \]
	where as before we always assume $Q_1(t) \in t \Q[t]$ and $Q_2(t)\dots,Q_l(t) \in \Q[t]$.  
	As we will see below $\zq$ is the space in which we are interested the most. 
	For $d\geq 0$ we define the subspace
	$ \zq[d] = \Big \langle   \zeta_q(s_1,\dots,s_l ; Q_1 , \dots , Q_l) \in \zq \,\,\big|\,\,\deg(Q_j) \leq s_j -d
	\Big\rangle_\Q$. 
	So in particular it is $\zq = \zq[0]$ and $\zq[d+1] \subset \zq[d]$.
	We also restrict to the case in which all polynomials $Q_j$  (not just $Q_1$) have no constant terms and therefore are elements in $t\Q[t]$. The resulting space is 
	denoted by
	\[ \zqc = \Big \langle   \zeta_q(s_1,\dots,s_l ; Q_1 , \dots , Q_l) \in \zq \,\,\big|\,\, Q_1,\dots,Q_l \in t \Q[t]
	\Big\rangle_\Q\,. \]
	For the  spaces $\zqc[d]$ given by  $\zqc \cap \zq[d]$  
	it holds  $\zqc = \zqc[0]$ and $\zqc[d+1] \subset \zqc[d]$.  
	Notice that all of these spaces are closed under multiplication. In depth one for example it is
	\[ \zeta_q(s_1; Q_1) \cdot \zeta_q(s_2; Q_2) = \zeta_q(s_1,s_2; Q_1,Q_2) + \zeta_q(s_2,s_1; Q_2,Q_1) +\zeta_q(s_1+s_2; Q_1 \cdot Q_2)  \,,\]
	and clearly $\deg Q_1 \cdot Q_2 \leq s_1+s_2-d$ if $\deg Q_j \leq s_j-d$ for $j=1,2$. 
	
	In \cite{Zh} Zhao considers for $s_1,\dots,s_l,d_1,\dots,d_l \in \Z$ the series
	\begin{equation} \label{eq:zhaozeta} \mathfrak{z}^{d_1,\dots,d_l}_q(s_1,\dots,s_l) = \sum_{n_1 > \dots > n_l >0} \frac{q^{n_1 d_1} \dots q^{n_l d_l}}{(1-q^{n_1})^{s_1}\cdots (1-q^{n_l})^{s_l}} \,,\end{equation}
	which gives an even more general setup than our $\zeta_q$. Especially these series can be seen as natural generators of the spaces $\zq[d]$ and $\zqc[d]$ by choosing the 
	appropriate
	 conditions on the $d_j$.
	We will now give a short overview of different $q$-analogues of multiple zeta values, which can be written in terms of the $\zeta_q$ and relate them to the spaces $\zq[d]$ and $\zqc[d]$.
	
	\begin{enumerate}[(i)]
		\item	The space spanned by the Bradley-Zhao model $\zeta^{\text{BZ}}_q=\zeta_q(s_1,\dots,s_l; t^{s_1-1},\dots,t^{s_j-1})$, defined in \eqref{eq:defzetaq}, is given by\footnote{
			This follows easily from the fact that 
			$t^{j-1} (1-t)^{s-j}$ with $j=1,\dots,s$ (resp. $j=2,\dots,s$) forms a basis of $\{ Q \in \Q[t] \mid \deg Q \leq s-1 \} $ (resp. $\{ Q \in t \Q[t] \mid \deg Q \leq s-1 \} $).
		} 
		\[
		\zq[1] = \big \langle \zeta^{\text{BZ}}_q(s_1,\dots,s_l) \mid l\geq 0 \,, s_1 \geq 2, s_2,\dots,s_l \geq 1  \big \rangle_Q\,.
		\]
		
		\item Another interesting case  is the Schlesinger-Zudilin model. These q-analogues are for $s_1\geq 1, s_2,\dots,s_l \geq 0$   defined by
		\begin{align} \label{eq:schlesingerzudiling}
		\begin{split}
		\zeta^{\text{SZ}}_q(s_1,\dots,s_l) &= \sum_{n_1 > \dots > n_l >0} \frac{q^{n_1 s_1} \dots q^{n_l s_l}}{(1-q^{n_1})^{s_1}\cdots (1-q^{n_l})^{s_l}} \\
		&=\zeta_q(s_1,\dots,s_l; t^{s_1},\dots,t^{s_j}) \,. 
		\end{split}
		\end{align}
		The space spanned by these series is, using the same argument as in (i), given by
		\[	
		\zq = \big \langle  	\zeta^{\text{SZ}}_q(s_1,\dots,s_l) \, \big |  \, l \ge 0, \, s_1\geq 1, s_2,\dots,s_l \geq 0 \big\rangle_\Q\,.
		\]
		Originally defined by Schlesinger \cite{Sch} and Zudilin \cite{Zu2} for the cases $s_1 \geq 2, s_2,\dots,s_l \geq 1$, it was observed in \cite{S} and further discussed in \cite{ems} that the algebraic setup, especially the shuffle product analogue, for this model can be described nicely by allowing $s_1\geq 1, s_2,\dots,s_l \geq 0$.
		Restricting to $s_1,\dots,s_l\geq 1$ we get the subspace
		\[	
		\zqc = \big \langle  	\zeta^{\text{SZ}}_q(s_1,\dots,s_l) \, \big |  \, l \ge 0, \, s_1, \dots,s_l \geq1 \big\rangle_\Q\,.
		\]
		\item In \cite{YOZ} Ohno-Okuda-Zudilin define for  $s_1,\dots,s_l \in \Z$  the series
		\begin{align} \label{eq:ooz}
		\zeta^{\text{OOZ}}_q(s_1,\dots,s_l) &= \sum_{n_1 > \dots > n_l >0} \frac{q^{n_1}}{(1-q^{n_1})^{s_1}\cdots (1-q^{n_l})^{s_l}} \,.
		\end{align}
		In the case $s_1,\dots,s_l \geq 1$ these can be written as $\zeta_q(s_1,\dots,s_l; t,1,\dots,1) \in \zq$, but  the space spanned  by \eqref{eq:ooz} for $s_1,\dots,s_j\geq 1$ is a priori not given by one of the $\zq[d]$ or $\zqc[d]$.  
		
		\item For $s_1,\dots,s_l \geq 2$ Okounkov chooses the following polynomials in \cite{O}
		\[  Q^O_j(t) = \begin{cases} t^{\frac{s_j}{2}} & s_j = 2,4,6,\dots  \\ 
		t^{\frac{s_j-1}{2}} (1+t) & s_j=3,5,7,\dots . \end{cases} \]
		and defines $\OZ(s_1,\dots,s_l) =  \zeta_q(s_1,\dots,s_l; Q^O_1,\dots,Q^O_l)$. With the same arguments as before (see also the proof of Theorem \ref{thm:bdasqmzv} (iii)) the span of these series is given by 
		\[	
		\zqc[1] = \big \langle  \OZ(s_1,\dots,s_l) \, \big |  \, l \ge 0, \, s_1,\dots,s_l \geq 2 \big\rangle_\Q\,.
		\]
	\end{enumerate}

	Although the space $ \zq$ seems to be much larger than the space $\zqc$, we expect that they both coincide (Conjecture \ref{conj:bconj} (B2) below) and therefore every $\zeta^{\text{SZ}}_q$ should be expressible as a linear combination of $\zeta^{\text{SZ}}_q(s_1,\dots,s_l)$ with  $s_1,\dots,s_l \geq 1$. 
	In \cite{ems} (Theorem 5.5) such an expression
	for $\zeta^{\text{OOZ}}_q$  in terms of $\zeta^{\text{SZ}}_q$ is given, which in turn can be seen as a special case of that conjecture.

	\begin{remark} As seen in the example above, the polynomials $Q_j$ often depend just on $s_j$. For these types of $q$-analogues one can also define subspaces of $\zq$ in the following way: 
		Suppose that $\{Q_s\}_{s\geq 1}$ is a family of polynomials, where for all $s_1,s_2 \geq 1$ there exists numbers $\lambda_j^{s_1,s_2} \in \Q$ with $j\geq 1$ and $\lambda_j^{s_1,s_2} = 0$ for almost all $j$, such that 
		\[ Q_{s_1}(t) \cdot Q_{s_2}(t) = \sum_{j=1}^\infty \lambda_j^{s_1,s_2} Q_j(t) (1-t)^{s_1+s_j-j} \,.\]
		Then the space spanned by all $\zeta_q(s_1,\dots,s_l; Q_{s_1},\dots,Q_{s_l})$ is a sub-algebra of $\zq$.
		This also gives an example of a so called quasi-shuffle algebra as described in \cite{HoffmanIhara}. For this one can define for $a,b\geq 1$ the product $z_a \diamond z_b = \sum_{j=1}^\infty \lambda_j^{a,b} z_j$ with the same notation as used in the first section of \cite{HoffmanIhara}. This was for example done in \cite{BK} for the space $\zqc$.
	\end{remark}

	\subsection{Bi-brackets as $q$-analogues of MZVs}

	In this section we will consider the $q$-series from the introduction in more detail and explain their connection to $q$-analogues of multiple zeta values in the section before. 
	
	\begin{definition}\begin{enumerate}[(i)]
			\item For $s_1,\dots,s_l \geq 1$, $r_1,\dots,r_l \geq 0$ we define the following $q$-series
			\begin{align*}
			\mb{s_1, \dots , s_l}{r_1,\dots,r_l} :=\sum_{\substack{u_1 > \dots > u_l> 0 \\ v_1, \dots , v_l >0}} \frac{u_1^{r_1}}{r_1!} \dots \frac{u_l^{r_l}}{r_l!} \cdot \frac{v_1^{s_1-1} \dots v_l^{s_l-1}}{(s_1-1)!\dots(s_l-1)!}   \cdot q^{u_1 v_1 + \dots + u_l v_l} \in \Q[[q]]\,.
			\end{align*}
			We refer to these  $q$-series as \emph{bi-brackets} of depth $l$ and of
			weight $s_1+\dots+s_l+r_1+\dots+r_l$.
			\item	For $r_1=\dots=r_l=0$ we write
			\[  [s_1,\dots,s_l] :=  \mb{s_1, \dots , s_l}{0,\dots,0}  \,. \]  
			The series  $[s_1,\dots,s_l]$, which we call \emph{brackets}, were introduced and studied in \cite{BK}.
		\end{enumerate}
	\end{definition}

	The bi-brackets also have an alternative form, which we will use now. For this recall that the Eulerian polynomials (See for example \cite{df}) are defined by
	\[ \frac{t P_{s-1}(t)}{(1-t)^{s}} = \sum_{d=1}^\infty d^{s-1} t^d  \,.\] 
	For $s>1$ the polynomials $t P_{s-1}(t)$ have degree $s-1$ and in the case $s=1$ it is $t P_{0}(t) = t$. 
	By definition of the bi-brackets it is then clear that 
	\begin{equation} \label{eq:bibracinP}
	\mb{s_1, \dots , s_l}{r_1,\dots,r_l} = \sum_{n_1 >\dots > n_l >0} \prod_{j=1}^l  \left( \frac{n_j^{r_j}}{r_j!} \cdot \frac{q^{n_j} P_{s_j-1}(q^{n_j}) }{(s_j-1)! \cdot (1-q^{n_j})^{s_j}} \right) \,. 
	\end{equation}
	
	We will now see that the spaces spanned\footnote{
	 In the  articles \cite{Ba2}, \cite{Ba4} and \cite{BK} these spaces   were denoted $\bMD$ and $\MD$.
	}
	 by the bi-brackets and brackets are exactly given by the spaces $\zq$ and $\zqc$ respectively. 
	
	\begin{theorem}\label{thm:bdasqmzv} The following equalities hold
		\begin{enumerate}[(i)]
			\item \hfill$\begin{aligned}[t]\zq &= \Big \langle \mb{s_1, \dots , s_l}{r_1,\dots,r_l} \,\, \big| \,\, l\geq 0 , s_1,\dots,s_l\geq 1,\, r_1,\dots,r_l \geq 0 \Big \rangle_\Q \,.\end{aligned}$\hfill\null
			\item \hfill$\begin{aligned}[t] \zqc &= \big \langle [s_1,\dots,s_l] \mid l\geq 0 \,, s_1,\dots,s_l \geq 1  \big \rangle_Q \,.\end{aligned}$\hfill\null
			\item\hfill $\begin{aligned}[t]\zqc[1] &= \big \langle [s_1,\dots,s_l] \mid l\geq 0 \,, s_1,\dots,s_l \geq 2  \big \rangle_Q \,. \end{aligned}$\hfill\null
		\end{enumerate}
	\end{theorem}
	\begin{proof} 
		Since for all $s\geq 1$ it is $P_{s-1}(1) \neq 0$  the polynomials $t P_{j-1}(t) (1-t)^{s-j}$ with $j=1,\dots,s$ form a basis of the space $\{ Q \in t \Q[t] \mid \deg Q \leq s \} $. In particular for every polynomial $Q$ in this space there exists coefficients $\alpha_j \in \Q$ with 
		\begin{equation} \label{eq:QasP}
		\frac{Q(t)}{(1-t)^s} = \sum_{j=1}^s \alpha_j \frac{t P_{j-1}(t)}{(1-t)^j}\,,
		\end{equation}
		from which the statement (ii) follows. Also (iii) follows, since for $d=1$ the condition $\Q_j(t) \in t \Q[t]$ and $\deg Q_j \leq s_j -1$ implies $s_j \geq 2$ for all $j=1,\dots,l$.  One can also see that 
		\[\big \langle [s_1,\dots,s_l] \mid l\geq 0 \,, s_1,\dots,s_l \geq 2  \big \rangle_Q = \big \langle \zeta^{\text{BZ}}_q(s_1,\dots,s_l)  \mid l\geq 0 \,, s_1,\dots,s_l \geq 2  \big \rangle_Q \,.\]

		To prove (i) we will first show the inclusion '$\subseteq$' , i.e. that every $\zeta_q(s_1,\dots,s_l;Q_1,\dots,Q_l)$ can be written in terms of bi-brackets. For this we need to see what happens if one of the $Q_2,\dots,Q_l$ has a constant term. Without loss of generality we can, by the proof of (ii), focus on the cases $Q_i(t) = 1$ for a $2 \leq i \leq l$ .  Since for all $s\geq 1$ it is
		\[ \frac{1}{(1-t)^s} = 1 + \sum_{m=1}^s \frac{t}{(1-t)^m}\,,\]
		we can write 
		\begin{align*} \sum_{n_1>\cdots>n_l>0} \prod_{j=1}^l \frac{Q_j(q^{n_j})}{(1-q^{n_j})^{s_j}} &=  \sum_{n_1>\cdots>n_l>0} \prod_{\substack{j=1\\ j \neq i}}^{l} \frac{Q_j(q^{n_j})}{(1-q^{n_j})^{s_j}} +  \sum_{\substack{n_1>\cdots>n_l>0\\ 1 \leq m \leq s_i}} \frac{q^{n_i}}{(1-q^{n_i})^m} \prod_{\substack{j=1\\ j \neq i}}^{l} \frac{Q_j(q^{n_j})}{(1-q^{n_j})^{s_j}} \,.
		\end{align*}
		For the the second sum on the right-hand side we can again use \eqref{eq:QasP}. For the first sum we obtain (by setting $n_{l+1} =0$)
		\begin{align*}
		\sum_{n_1>\cdots>n_l>0} \prod_{\substack{j=1\\ j \neq i}}^{l} \frac{Q_j(q^{n_j})}{(1-q^{n_j})^{s_j}} = \sum_{n_1>\cdots>n_{i-1} > n_{i+1} > \dots >n_l>0} (n_{i-1} - n_{i+1}-1) \prod_{\substack{j=1\\ j \neq i}}^{l} \frac{Q_j(q^{n_j})}{(1-q^{n_j})^{s_j}}\,.
		\end{align*}
		Repeating this for all  $2 \leq i \leq l$ with $Q_i(t) = 1$ we obtain sums of the form \eqref{eq:bibracinP} from which we deduce '$\subseteq$'.
		
		Now to prove '$\supseteq$' we first define for $m\geq 0$ the polynomials $p_m(n)$ by $p_0(n)=1$ and 
		\begin{equation}\label{eq:defpm}
		p_m(n) = \sum_{n > N_1 > \dots > N_m > 0} 1 \,. 
		\end{equation}
		The $p_m(n)$ is a polynomial in $n$ of degree $m$ and therefore we can always find $c_m(r) \in \Q$ with $n^r = \sum_{m=0}^r c_m(r)\, p_m(n)$. The idea is now to replace $n_j^{r_j}$ in the definition of the bi-brackets by $\sum_{m_j=0}^{r_j} c_{m_j}(r_j)\, p_{m_j}(n_j)$ and then use \eqref{eq:defpm} to get sums which can be written in terms of the $\zeta_q$. We illustrate this in the depth two case from which the general case becomes clear. 
		We have with $\kappa= (s_1-1)!(s_2-1)!r_1!r_2!$
		\begin{align*}
		\kappa \cdot \mb{s_1,s_2}{r_1,r_2} &= \sum_{n_1 > n_2 > 0} \frac{n_1^{r_1} q^{n_1} P_{s_1-1}(q^{n_1})}{(1-q^{n_1})^{s_1}} \frac{n_2^{r_2} q^{n_2} P_{s_2-1}(q^{n_2})}{(1-q^{n_2})^{s_2}} \\
		&= \sum_{0 \leq m_2 \leq r_2 } c_{m_2}(r_2) \sum_{n_1 > n_2 > N_1 > \dots > N_{m_2}> 0} \frac{n_1^{r_1} q^{n_1} P_{s_1-1}(q^{n_1})}{(1-q^{n_1})^{s_1}} \frac{ q^{n_2} P_{s_2-1}(q^{n_2})}{(1-q^{n_2})^{s_2}} \\
		&= \sum_{\substack{0 \leq m_1 \leq r_1 \\0 \leq m_2 \leq r_2 }} c_{m_1}(r_1) c_{m_2}(r_2) \sum_{\substack{n_1 > n_2 > N_1 > \dots > N_{m_2}> 0 \\ n_1 >  N'_1 > \dots > N'_{m_1}> 0}} \frac{q^{n_1} P_{s_1-1}(q^{n_1})}{(1-q^{n_1})^{s_1}} \frac{ q^{n_2} P_{s_2-1}(q^{n_2})}{(1-q^{n_2})^{s_2}}\,.
		\end{align*}
		Now considering all the possible shuffles, and possible equalities of the $N$ and the $N'$ it is clear that this sum can be written as a linear combination of $\zeta_q$ by interpreting appearing $1$ as $(1-q^N) (1-q^N)^{-1}$. For general depth $l$ the idea is the same and therefore we obtain '$\supseteq$' from which (i) follows.  
	\end{proof}
	
	As an example how to write a bi-bracket in terms of $\zeta_q$ we give the following.
	\begin{align}\label{example-fildefect} 
	\begin{split}
	\mb{1,1}{0,1} &= \sum_{n_1 > n_2 > 0} \frac{q^{n_1}}{(1-q^{n_1})}  \frac{n_2 q^{n_2}}{(1-q^{n_2})} \\&= \sum_{n_1 > n_2 > 0} \frac{q^{n_1}}{(1-q^{n_1})}  \frac{q^{n_2}}{(1-q^{n_2})} + \sum_{n_1 > n_2 > n_3 >0} \frac{q^{n_1}}{(1-q^{n_1})}  \frac{q^{n_2}}{(1-q^{n_2})} \frac{1-q^{n_3}}{(1-q^{n_3})}\\
	&= \zeta_q(1,1; t, t) +  \zeta_q(1,1,1;t,t,1-t)\,.
	\end{split}
	\end{align}

	\subsection{Bi-brackets and quasi-modular forms} 
	
	We now define the weight and the depth filtration for the space $\zq$ by writing for a subset $A\subseteq \zq$
	\begin{align*} 
	\filw_k(A) &:=  \big< \mb{s_1,\dots,s_l}{r_1, \dots, r_l} \in A \,\big|\, 0 \leq l \leq k \,,\, s_1+\dots+s_l+r_1+\dots+r_l \le k \,\big>_{\Q}\\
	\fild_l(A) &:=  \big<\mb{s_1,\dots,s_t}{r_1, \dots, r_t} \in A \,\big|\, t\le l \,\big>_{\Q}\,.
	\end{align*}
	If we consider the depth and weight filtration at the same time we use the short notation $\filwle_{k,l} := \filw_k \fild_l$ and similar for the other filtrations. The associated graded spaces will be denoted by $\grw_k$ and $\grwl_{k,l}$.
	
	\begin{remark}\begin{enumerate}[(i)]
			\item We point to the fact that the filtration by depth coming from bi-brackets is different from the naive notion of depth for the $\zeta_q(s_1,\dots,s_l)$, given as the number of variables $s_i$. For example, as indicated by  \eqref{example-fildefect}, the $\zeta_q(1,1,1;t,t,1-t)$ is an element in $\fild_2(\zq)$.
			\item  As seen before the Schlesinger-Zudilin model $	\zeta^{\text{SZ}}_q(s_1,\dots,s_l)$, defined in \eqref{eq:schlesingerzudiling} for $s_1\geq 1, s_2,\dots,s_l \geq 0$, span the space $\zq$ and therefore we also obtain a depth and weight filtration for these series. By the proof of Theorem \ref{thm:bdasqmzv} we see that  $\zeta^{\text{SZ}}_q(s_1,\dots,s_l) \in \filwle_{K,L}(\zq)$ with $K = s_1+\dots+s_l + z$ and $L = l + z$, where $z = \#\{ j \mid s_j = 0\}$ is the number of $s_j$ which are zero.
		\end{enumerate} 
	\end{remark} 
	
	For several reasons one should consider these filtrations to be the natural ones. 
	First of all the multiplication in $\zq$ respects the depth as well as the weight grading. Secondly, on $\zq$ we have the derivation given by $q\frac{d}{dq}$, it increases the weight by $2$ and 
	keeps the depth, since we obtain directly from the definition that
	\[ 
	q\frac{d}{dq} \mb{s_1, \dots , s_l}{r_1,\dots,r_l} = \sum_{j=1}^l \left( s_j (r_j+1) \mb{s_1\,,\dots\,,s_{j-1}\,,s_j+1\,,s_{j+1},\dots \,,s_l}{r_1 \,, \dots \,,r_{j-1}\,,r_j+1\,,r_{j+1} \,,\dots \,, r_l} \right) \, .
	\] 
	Thirdly, the  classical Eisenstein series are contained in $\zqc \subset \zq$. For example we have
	\[ {G}_2 = -\frac{1}{24} + [2] \,,\quad {G}_4 = \frac{1}{1440} + [4] \,, \quad {G}_6 = -\frac{1}{60480} + [6] \,,\]
since in depth one it is for $k>0$
	\[ [k] = \sum_{\substack{u>0\\v>0}} \frac{v^{k-1}}{(k-1)!} q^{u v} = \frac{1}{(k-1)!} \sum_{n>0} \sum_{d \mid n} d^{k-1}  q^n = \frac{1}{(k-1)!} \sum_{n>0}\sigma_{k-1}(n) q^n \,.\]

	The space of quasi-modular forms for $\operatorname{SL}_2(\Z)$ with rational coefficients is given by $\widetilde{M} (\operatorname{SL}_2(\Z))_\Q  = \Q[G_2,G_4,G_6]$  (see\cite{KaZa}) and therefore it is a sub-algebra of $\zqc$ and $\zq$. 
	It is graded by the weight, in the classical sense, and 
	obviously   $\widetilde{M}_k (\operatorname{SL}_2(\Z))_\Q  \subset \filw_k(\zq)$. 
	The derivation $ q\frac{d}{dq}$ increases the weight by $2$, i.e.
	\[
	q\frac{d}{dq} : \widetilde{M}_k(\operatorname{SL}_2(\Z))_\Q \to \widetilde{M} _{k+2} (\operatorname{SL}_2(\Z))_\Q.
	\]
The space of quasi-modular forms has the decomposition 
	\begin{align}\label{qmf:decomp}
	\widetilde{M}_k (\operatorname{SL}_2(\Z))_\Q  
	= \big\langle G_k , \, q\frac{d}{dq} G_{k-2},\dots, \,\big(q\frac{d}{dq}\big)^{k/2-1} G_2 \big\rangle_\Q \oplus
	\bigoplus_{i=0}^{k/2}  \,\, \big(q\frac{d}{dq}\big)^i S_{k-2 i}(\operatorname{SL}_2(\Z))_\Q\,.
	\end{align}

	\begin{proposition}\label{prop:qmf} Set
		$
		\tM(x,t) = 1 + \Dx(x) \ev_2(x)\, t +  \Dx(x) \Sx(x)\,  t^2
		$,
		then the generating series 
		for the weight- and depth-graded dimensions of $ \widetilde{M} (\operatorname{SL}_2(\Z))_\Q \subset \zq$ 
		satisfies the coefficient-wise inequality
		\begin{align}\label{qmf-generating-series}
		\sum_{k,l} \dim_\Q  \grwl_{k,l}(\widetilde{M}(\operatorname{SL}_2(\Z))_\Q ) x^k t^l 
		\le
		\tM(x,t)   \,.	 
		\end{align}
		
	\end{proposition}
	\begin{proof}
		The Eisenstein series and their derivatives are in the depth one subspaces.   For the space of cusp forms of weight $k$ we have, by using a theorem of Zagier,
		\[ S_{k}(\operatorname{SL}_2(\Z))_\Q 
		\subset \big \langle   G_{k-a} G_a \,| a=0,\dots,\,k/2 \,\big \rangle_\Q 
		\subset
		\filwle_{k,2}(\zq).
		\] 
		Finally since $q\frac{d}{dq}$ does not alter the depth we get the claim by the decomposition \eqref{qmf:decomp}.
	\end{proof}
	
	The expected equality in Proposition \ref{prop:qmf} would hold if the brackets $[2],\,[4],\,[6]$ and the odd brackets $[1],\,[3],..$ together with all of their derivatives were algebraically independent, but by now only partial results for linear independence are available (\cite{pupyrev},\cite{Zu3}). 
		
	\begin{conjecture}  \label{conj-decomposition}
		We have a decomposition of $\mathbb{Q}$-algebras
		\[
		\zq \cong  \widetilde{M}_\Q(\operatorname{SL}_2(\Z))  \otimes \calA.
		\]
		This decomposition is respected by the operator $q\frac{d}{dq}$. Moreover		
		$\calA$ is a free polynomial algebra that is bi-graded with respect to weight and depth compatible with those of $\zq$. In particular it equals the  
		 graded dual to the universal enveloping algebra of a bi-graded Lie algebra\footnote{Some authors prefer to denote this as the symmetric algebra of a Lie algebra}.
	\end{conjecture}
	
This decomposition of algebras should be seen as an analogue of   \cite[ Conjecture 1.1.~b)]{Go}  
	in our context. 
	The conjecture above implies the weaker claim, that the algebra $\zq$ is isomorphic to a free polynomial algebra graded by the weight. It also implies that in Proposition \ref{prop:qmf} 
	the equality holds.
		
	\begin{remark}	In \cite{O} Okounkov gives the following conjecture for the dimension of the weight-graded parts of $\zqc[1]$. 
		\begin{align}\label{eq:okdim}
		\!\!\!\!\!\!\!\!\!  \sum_{k\geq 0}  \dim_\Q \left( \grw_k \zqc[1] \right)   \, x^k \overset{?}{=} \frac{1}{ 1 -x^2 -x^3 -x^4 -x^5 + x^8 +x^9+  x^{10}+  x^{11} +  x^{12} }.
		\end{align}
		We expect that the decomposition of  Conjecture \ref{conj-decomposition} 
		induces also a decomposition for $\zqc[1]$. Indeed, keeping the previous notation, this is compatible with the factorization   
		\[	
		\frac{1}{ 1 -x^2 -
		\ldots
		-x^5 + x^8 +
		\ldots +  x^{12} }  
		=  \tM(x,1) \frac{1} { 1-  \Dx(x) \odd_3(x) + 2 \Dx(x)\Sx(x)  }  \,.
		\]
		Our Conjecture \ref{conj:bkforbibra} (i) for $\zq$ yields with $\ev_4(x)= x^4/(1-x^2)$
		\[
		\frac{1}{ 1 -x -x^2-x^3 + x^6 +
		\ldots
		+x^9 } 
		=  \tM(x,1) \frac{1} { 1 - \Dx(x)   \odd_1(x) + \Dx(x) 
			\big( \ev_4(x) + 2 \Sx(x)\big)  }. 
		\]
		Thus we may think of the Lie algebra behind $\zq$ compared to that behind $\zqc[1]$ as having additional generators induced by the derivatives of a generator in weight $1$ and having additional relations being counted by the number of Eisenstein series for $\operatorname{SL}_2(\Z)$ and their derivatives.	
	\end{remark}

	\section{Computational evidences for the conjectures}
	
	In this section we want to describe how to implement the bi-brackets to obtain the numerical results, which were used to obtain Conjecture \ref{conj:bkforbibra} in the introduction and further conjectures stated below.  A similar method
	to perform such calculations has been communicated to us by Don Zagier.
	
	Using \eqref{eq:bibracinP} we define for a fixed $N\in \N$  an approximated version of bi-brackets by
	\begin{align}\label{def:brf}
	\mb{s_1, \dots , s_l}{r_1,\dots,r_l}_N := \sum_{N\geq
		n_1>\cdots>n_l>0} \prod_{j=1}^l  \left( \frac{n_j^{r_j}}{r_j!} \cdot \frac{q^{n_j}  P_{s_j-1}(q^{n_j}) }{(s_j-1)! \cdot (1-q^{n_j})^{s_j}} \right)  \in\Q[[q]].
	\end{align}
	Observe that \( \mb{s_1, \dots , s_l}{r_1,\dots,r_l}_N= 0 \) for   \( N<l \). It is clear that at least the first $N$ coefficients of these approximated versions are identical to the bi-brackets, i.e.
	\[ 
	\mb{s_1, \dots , s_l}{r_1,\dots,r_l}_N \equiv  \mb{s_1, \dots , s_l}{r_1,\dots,r_l} \mod q^{N+1}.
	\]
	
	To calculate the first $N$ coefficients of the bi-brackets we use the following recursive formula for these approximated versions
	\begin{lemma}\label{lem:approx_recursion}
		For all $s_1,\dots,s_l, r_1, \dots , r_l$ and $N\geq l$ we have
		\begin{align*}
		\mb{s_1, \dots , s_l}{r_1,\dots,r_l}_N=   \mb{s_1, \dots , s_l}{r_1,\dots,r_l}_{N-1}
		+\frac{N^{r_1}}{r_1!}  \frac{q^N P_{s_1-1}(q^N)}{(s_1-1)!\cdot{(1-q^N)}^{s_1}}    \mb{s_2, \dots , s_l}{r_2,\dots,r_l}_{N-1} \,,
		\end{align*}
		where we set $\mb{s_2, \dots , s_l}{r_2,\dots,r_l}_{N-1}=1$ for $l=1$. 
	\end{lemma}
	\begin{proof}This follows by splitting up the summation $N\geq
		n_1>\cdots>n_l>0$ into the parts where $N > n_1$ and $N=n_1$ to get the first and the second term respectively. 
	\end{proof}

	We implemented an algorithm based on Lemma \ref{lem:approx_recursion} in parallel PARI/GP  \cite{PARI2} and on a computer with 32 cores  
	it takes several hours to obtain each of the following tables:

	\begin{table}[H]\footnotesize
		\resizebox{0.9\textwidth}{!}{
			\begin{tabular}  
				{c|cccccccccccccc} $k \backslash l$&1&2&3&4&5&6&7&8&9&10&
				11&12&13&14\\ \hline
				\noalign{\medskip}1&2&0&0&0&0&0&0&0&0&0&0&0&0&0
				\\ \noalign{\medskip}2&3&4&0&0&0&0&0&0&0&0&0&0&0&0
				\\ \noalign{\medskip}3&5&7&8&0&0&0&0&0&0&0&0&0&0&0
				\\ \noalign{\medskip}4&7&12&14&15&0&0&0&0&0&0&0&0&0&0
				\\ \noalign{\medskip}5&10&19&25&27&28&0&0&0&0&0&0&0&0&0
				\\ \noalign{\medskip}6&13&30&41&48&50&51&0&0&0&0&0&0&0&0
				\\ \noalign{\medskip}7&17&44&68&81&89&91&92&0&0&0&0&0&0&0
				\\ \noalign{\medskip}8&21&65&106&138&153&162&164&165&0&0&0&0&0&0
				\\ \noalign{\medskip}9&26&90&167&223&264&281&291&293&294&0&0&0&0&0
				\\ \noalign{\medskip}10&31&126&249&366&439&490&509&520&522&523&0&0&0&0
				\\ \noalign{\medskip}11&37&167&376&571&738&830&892&913&925&927&928&0&0
				&0\\ \noalign{\medskip}12&43&222&537&905&1190&1418&1531&1605&1628&1641
				&1643&1644&0&0\\ \noalign{\medskip}13&50&285&778&1364&1948&2344&2645&
				2781&2868&2893&2907&2909&2910&0\\ \noalign{\medskip}14&57&368&1075&
				2090&3051&3923&4453&4840&5001&5102&5129&5144&5146&5147  \end{tabular} 
		}
		\caption{lower bounds $\operatorname{fil}^{\num}_{k,l}(\zq)$ for $\dim_\Q  \filwle_{k,l}(\zq)$ with   depth $\le 14$}

	\end{table}
	
	\begin{table}[H]\footnotesize
		\resizebox{1.0\textwidth}{!}{
			\begin{tabular}{c|ccccccccccccccccccccc} $l \backslash k$&1&2&3&4&5&6&7&
				8&9&10&11&12&13&14&15&16&17&18&19&20&21\\ \noalign{\medskip}1&2&3&5&7&
				10&13&17&21&26&31&37&43&50&57&65&73&82&91&101&111&122
				\\ \noalign{\medskip}2&0&4&7&12&19&30&44&65&90&126&167&222&285&368&460
				&577&706&866&1041&1254&1485\\ \noalign{\medskip}3&0&0&8&14&25&41&68&
				106&167&249&376&537&778&1075&1503&2017&2737&3584&4739&6077&7859
				\\ \noalign{\medskip}
				4& 
				0 & 0& 0&15& 27&48& 81& 138& 223& 366& 571& 905& 1364&2090& 3053& 4535& 6440& 9293 &
				?&?&? 
			\end{tabular} 
		}
		\caption{lower bounds $\operatorname{fil}^{\num}_{k,l}(\zq)$ for $\dim_\Q  \filwle_{k,l}(\zq)$ with   depth $\le 4$} 
	\end{table}
	
	In fact for these tables we calculated approximated  bi-brackets with coefficients modulo some large prime and determined the  dimension they span at least. Experimentally the choice of a sufficiently large prime does not alter the dimension. We have similar tables for various subspaces like  the \emph{positive bi-brackets}
	\[
	\zq^+ = \big \langle   \mb{s_1, \dots , s_l}{r_1,\dots,r_l}  \in \zq \,   \big| \, 
	l \ge 0,\, s_1 > r_1, \dots,\, s_l>r_l  
	\big\rangle_\Q  
	\]
	or the space of \emph{$123$-brackets} given by
	\[
	\big \langle  [s_1,\dots,s_l]  \,\big | \, l \ge 0, \, s_1,\dots,s_l \in \{1,2,3\} \,\big\rangle_\Q \subset \zqc 
	\] 
	and for sub-algebras like $\zqc$ or $\zqc[1]$. This lead us to the following conjectures
	
	\begin{conjecture}\phantomsection\label{conj:bconj}
		\begin{enumerate} 
			\item[(B1)] Every bi-bracket equals a linear combination of positive bi-brackets		
			\item[(B1*)]	More precisely,  the space of positive bi-brackets $\zq^{+}$  
			fulfills
			$ \operatorname{Fil}^{\operatorname{W}, \operatorname{L}}_{w,l}(\zq^{+})  
			= \operatorname{Fil}^{\operatorname{W}, \operatorname{L}}_{w,l}(\zq) $. 
			\item[(B2)] Every bi-bracket equals a linear combination of brackets, i.e. $\zqc = \zq$. 
			\item[(B3)] Every bracket equals a linear combination of $123$-brackets. 
		\end{enumerate}
	\end{conjecture}
	Although, our experiments support conjectures (B1) and (B3), we were not able to prove the weaker claims that  the positive bi-brackets respectively the $123$-brackets generate  sub-algebras of $\zq$. In \cite{Ba2} the conjecture (B2) was stated the first time and therein examples which complement those in \cite{ems} (Theorem 5.5)  were given.

	\begin{theorem}\label{thm:upperBK}  For all weights  $k \le 7$ the coefficients 
	on both sides of Conjecture \ref{conj:bkforbibra} (i) coincide and in addition
	the Conjectures \ref{conj:bconj} (B1), (B2) and (B3) hold for these weights also. 
	\end{theorem}
	We will give a proof of this theorem at the end of this section.

	The Conjecture \ref{conj:bkforbibra} is based on the assumption that the above lower bounds were the actual dimensions. In other words, for the quantities
	\[
	\gr_{k,l}^{\num}   = \operatorname{fil}_{k,l}^{\num}(\zq) - \operatorname{fil}_{k,l-1}^{\num}(\zq) - \operatorname{fil}_{k-1,l}^{\num}(\zq)
	+ \operatorname{fil}_{k-1,l-1}^{\num}(\zq)
	\] 
	we expect the equalities $\gr_{k,l}^{\num}   = \dim_\Q \grwl_{k,l}(\zq)$. 
	Now we check if the generating series of the weight- and depth-graded parts of $\zq$ can be  of the shape implied by the conjectures.  	
	For example, if we assume that there is a decomposition $\zq \cong \widetilde{M}(\operatorname{SL}_2(\Z)) \otimes \mathcal{A}$, where the algebra
	$\mathcal{A}$  is a free polynomial algebra, then there must hold an equation of the form  
	\[
	\sum_{k,l \ge 0} \dim_\Q (\gr_{k,l} (\zq) ) \, x^k y^l  = \tM(x,y)
	\cdot \prod_{k,l \ge 1}  \frac{1}{ (1-x^k  y^l)^{g_{k,l}}  } \,,\]
	where the $g_{k,l}$  equal the number of generators of $\mathcal{A}$ in weight $k$ and depth $l$. 
	Solving such an equation with with $\gr_{k,l}^{\num} $ on the left-hand side, give us 
	numerical $g_{k,l}^{\num}$ and within the range of our experiments 
	(See Table \ref{table:freepoly} on page \pageref{table:freepoly}) these are positive and satisfy a parity pattern.
	
	\begin{table}[H] 
		\resizebox{1.0\textwidth}{!}{\footnotesize
			\begin{tabular}{c|cccccccccccccc} $\gr^{\num}_{k \backslash l}$&1&2&3&4&5&6&7&8&9&10&11
				&12&13&14\\ \hline \noalign{\medskip}1&1&0&0&0&0&0&0&0&0&0&0&0&0&0
				\\ \noalign{\medskip}2&1&1&0&0&0&0&0&0&0&0&0&0&0&0
				\\ \noalign{\medskip}3&2&1&1&0&0&0&0&0&0&0&0&0&0&0
				\\ \noalign{\medskip}4&2&3&1&1&0&0&0&0&0&0&0&0&0&0
				\\ \noalign{\medskip}5&3&4&4&1&1&0&0&0&0&0&0&0&0&0
				\\ \noalign{\medskip}6&3&8&5&5&1&1&0&0&0&0&0&0&0&0
				\\ \noalign{\medskip}7&4&10&13&6&6&1&1&0&0&0&0&0&0&0
				\\ \noalign{\medskip}8&4&17&17&19&7&7&1&1&0&0&0&0&0&0
				\\ \noalign{\medskip}9&5&20&36&24&26&8&8&1&1&0&0&0&0&0
				\\ \noalign{\medskip}10&5&31&46&61&32&34&9&9&1&1&0&0&0&0
				\\ \noalign{\medskip}11&6&35&86&78&94&41&43&10&10&1&1&0&0&0
				\\ \noalign{\medskip}12&6&49&106&173&118&136&51&53&11&11&1&1&0&0
				\\ \noalign{\medskip}13&7&56&178&218&299&168&188&62&64&12&12&1&1&0
				\\ \noalign{\medskip}14&7&76&214&429&377&476&229&251&74&76&13&13&1&1
				\end {tabular} 
				\hspace{20pt} 
				\begin {tabular}{c|cccccccccccccc} $g^{\num}_{k \backslash l}$&1&2&3&4&5&6&7&8&9&10&11
				&12&13&14\\ \hline \noalign{\medskip}1&1&0&0&0&0&0&0&0&0&0&0&0&0&0
				\\ \noalign{\medskip}2&0&0&0&0&0&0&0&0&0&0&0&0&0&0
				\\ \noalign{\medskip}3&2&0&0&0&0&0&0&0&0&0&0&0&0&0
				\\ \noalign{\medskip}4&0&1&0&0&0&0&0&0&0&0&0&0&0&0
				\\ \noalign{\medskip}5&3&0&1&0&0&0&0&0&0&0&0&0&0&0
				\\ \noalign{\medskip}6&0&2&0&1&0&0&0&0&0&0&0&0&0&0
				\\ \noalign{\medskip}7&4&0&3&0&1&0&0&0&0&0&0&0&0&0
				\\ \noalign{\medskip}8&0&7&0&3&0&1&0&0&0&0&0&0&0&0
				\\ \noalign{\medskip}9&5&0&8&0&4&0&1&0&0&0&0&0&0&0
				\\ \noalign{\medskip}10&0&12&0&11&0&4&0&1&0&0&0&0&0&0
				\\ \noalign{\medskip}11&6&0&22&0&14&0&5&0&1&0&0&0&0&0
				\\ \noalign{\medskip}12&0&20&0&31&0&17&0&5&0&1&0&0&0&0
				\\ \noalign{\medskip}13&7&0&47&0&44&0&21&0&6&0&1&0&0&0
				\\ \noalign{\medskip}14&0&31&0&81&0&58&0&25&0&6&0&1&0&0 
			\end{tabular} 
		}
		\caption{Evidence for $\mathcal{A}$ being a free polynomial algebra }
		\label{table:freepoly}
	\end{table}
	If we assume that there is a decomposition $\zq \cong \widetilde{M}(\operatorname{SL}_2(\Z)) \otimes \mathcal{A}$, where the algebra	$\mathcal{A}$  is the graded dual to the universal enveloping algebra of a bi-graded Lie algebra, then there must hold an equation of the form 
	\[
	\sum_{k,l \ge 0} \dim_\Q (\grwl_{k,l}(\zq))  \, x^k y^l  = \tM(x,y)
	\cdot    \frac{1}{ 1- \sum_{k,l\ge 1} b_{k,l}\, x^k  y^l   }  \]
	with $b_{k,l} \in \Z$.    
	Solving such an equation with with $\gr_{k,l}^{\num}(\zq)$ on the left-hand side, give us 
	numerical $b_{k,l}^{\num}$ and within the range of our experiments (See Table \ref{table:exprange} on page \pageref{table:exprange}) these are as expected in Conjecture \ref{conj:bkforbibra} (ii)   
	{\small
		\begin{table}[H]  
			\begin{center}
				\begin{tabular}{c|c|c|c|c|c|c|c|c|c|c|c|c|c|c } 
					$l$ &1&2&3&4&5&6&7&8&9&10&11&12&13&14\\ \hline
					$k \le$ &63&31&21&19&15&14&14&14&14&14&14&14&14&14 \\  
				\end{tabular}
				\caption{Evidence for $\mathcal{A}$ being a symmetric algebra of a Lie algebra  }
				\label{table:exprange}
			\end{center}
	\end{table}}

	Whereas it is known that the numbers from Zagier's conjecture give upper bound for the dimensions in question, the knowledge about the Broadhurst-Kreimer conjecture is very little. The only known results are the following:
	
	\begin{theorem}[Euler, Ihara-Kaneko-Zagier, Goncharov, Ihara-Ochiai]   For $1 \leq l \leq 3$ the numbers $g_{k,l}$ of generators for $\MZ$ of weight $k$ and depth $l$ are not bigger	than implied by the Broadhurst-Kreimer conjecture.
	\end{theorem}
	
	The proof of this result for $l=1$ is a trivial consequence of Euler's formula for even zeta values. For $l=2,3$ one can bound the number of generators by the dimension of the so called double shuffle spaces, see e.g. \cite{IKZ}, \cite{Go}, \cite{IO} . 

	We now want to use a similar technique to obtain upper bounds of the number of algebra generators for bi-brackets.
	
	For the generating function of the bi-brackets we write
	\begin{align*}
	\mt{ X_1,\dots,X_l}{Y_1,\dots,Y_l} := 
	\sum_{\substack{s_1,\dots,s_l > 0 \\ r_1, \dots, r_l > 0}} \mb{s_1\,,\dots\,,s_l}{r_1-1\,,\dots\,,r_l-1} X_1^{s_1-1} \dots X_l^{s_l-1} \cdot Y_1^{r_1-1} \dots Y_l^{r_l-1}. 
	\end{align*}
	
	As shown in \cite{Ba2} this satisfies the partition relation
	\[ \mt{X_1,\dots,X_l}{Y_1,\dots,Y_l} 
	= \mt{X_1,\dots,X_l}{Y_1,\dots,Y_l} \Bigg |_P \,,
	\]
	with
	$\small{
	f( X_1,\dots,X_l, Y_1,\dots,Y_l) \big |_P = f( Y_1+\dots+Y_l, \dots, Y_1+Y_2 , Y_1, 
	X_l, X_l-X_{l-1},\dots, X_2-X_1)}.
	$
	Up to terms of depth less than $l$
	their product is given by 
	\[ 
	\mt{X_1,\dots,X_j}{Y_1,\dots,Y_j}   \cdot \mt{X_{j+1},\dots,X_{l}}{Y_{j+1},\dots,Y_{l}} 
	=  \mt{X_1,\dots,X_l}{Y_1,\dots,Y_l}  \Bigg |_{\operatorname{Sh}_{j,l}} + \dots \,,
	\]
	where, if  $ \Sigma_{j,l} \subset \Sigma_n$ denotes the shuffles of ordered sets 
	with $j$ and $l-j$ elements, we have  
	\[
	f( X_1,\dots,X_l, Y_1,\dots,Y_l )\big |_{\operatorname{Sh}_{j,l}} 
	=\sum_{\sigma \in \Sigma_{j,l}} f( X_{\sigma^{-1}(1)},  \dots,X_{\sigma^{-1}(l)},Y_{\sigma^{-1}(1)},  \dots,Y_{\sigma^{-1}(l)})
	\]
	Hence we get modulo products  and lower depth bi-brackets  
	\[ 
	\mt{X_1,\dots,X_l}{Y_1,\dots,Y_l}  \equiv \sum_{\alpha } 
	\alpha F_\alpha ( X_1,\dots,X_l, Y_1,\dots, Y_l)\,, 
	\]
	where $\alpha$ runs through a vector space basis  of the depth $l$ algebra generators of $\zq$  and 
	$F_\alpha$ is a polynomial in the partition shuffle space, which is defined as follows.     
	
	\begin{definition}
		Define for $l,k \geq 0$ the  partition shuffle space  by
		\[
		\mathbb{PS}(k-l,l) = \{ f \in \Q[x_1,..,x_l,y_1,..,y_l] | 
		\deg f = k-l, \,\, f \big|_P -f = f \big|_{\operatorname{Sh}_j} =0 \,\, \forall j \}\,.
		\]
	\end{definition}
	
	Using the same argument as in \cite{IKZ} the above discussion leads to the following upper bounds.
	
	\begin{corollary}\label{ps-bounds}
		The number $g_{k,l}$ of generators of weight $k$ and depth $l$ for the $\Q$-algebra $\zq$ is bounded by
		\[
		g_{k,l} \le \dim_\Q \mathbb{PS}(k-l,l).
		\]
	\end{corollary}
	
	The bounds obtained via the partition shuffle spaces for the number of generators in depth $1$ and even weights are not optimal, as it is well-known that the ring of quasi-modular forms is generated in weight $2,4$ and $6$.  We view this as the analogue to the fact that Euler's relation for even zeta values is not seen by the depth $1$ double shuffle spaces as defined in \cite{IKZ}.
	
	{\small
		\begin{table}[H]  
			\begin{center}
				\begin{tabular}{c|c|c|c|c|c|c|c|c|c|c|c|c|c|c|c|c|c|} 
					$p_{l \backslash k}$&1&2&3&4&5&6&7&8&9&10&11&12&13&14&15&16&17\\ \hline
					1&1&1&2&2&3&3&4&4&5&5&6&6&7&7&8&8&9\\ \hline
					2&-&0&0&1&0&2&0& 8&0&14&0&23&0&38&0&58&0\\ \hline
					3&-&-&0&0&1&0&3&0&9&0&27&0&62&0&125&0&238\\ \hline
					4&-&-&-&0&0&1&0&3&0&12&0&37&?&?&?&?&?\\ \hline
					5& -&-&-&-&0&0&1&0&4&0&15&?&?&?&?&?&?\\ \hline
					6& -&-&-&-&-&0&0&1&?&?&?&?&?&?&?&?&? \\ \hline
				\end{tabular}
				\caption{ $p_{k,l} = \dim_\Q \mathbb{PS}(k-l,l)$  \label{tab:gwl-bounds}}
			\end{center}
	\end{table}}
	
	{\bf Proof of Theorem \ref{thm:upperBK}}  Using the structure of the ring of quasi-modular forms and the data of Table \ref{tab:gwl-bounds} we get the 
	coefficient-wise upper bounds   
	{\small
		\begin{align*}
		\sum_{k \ge 0} &\dim_\Q  \filw_{k}(\zq) x^k  \le \frac{1}{1-x} \frac{1}{(1-x^2)(1-x^4)(1-x^6)} 
		\frac{x}{(1-x^2)^2} \cdot
		\prod_{k,l \ge 2} \frac{1}{(1-x^k)^{p_{k,l}}} \\
		&\le 1+2\,x+4\,{x}^{2}+8\,{x}^{3}+15\,{x}^{4}+28\,{x}^{5}+51\,{x}^{6}+92\,
		{x}^{7}+ 166\,{x}^{8}+  \dots
		\end{align*}}
	In addition, since $\mbox{"123-brackets"} \subseteq \zqc \subseteq \zq$, we get by the data of our tables  
	\begin{align*}
	&1+2\,x+4\,{x}^{2}+8\,{x}^{3}+15\,{x}^{4}+28\,{x}^{5}+51\,{x}^{6}+92\,
	{x}^{7}+ 165\,{x}^{8}+  \dots \\
	&\le  \sum_{k\ge 0} \dim_\Q  \filw_{k}(\mbox{"123-brackets"}) x^k  \le  \sum_{k\ge 0} \dim_\Q  \filw_{k}(\zq) x^k .  
	\end{align*}
	The claim of the theorem follows as the lower  and upper bounds coincide for $k \le 7$.
	\hfill $\Box$
	\medskip
	
	\begin{remark}
		In contrast to the multiple zeta values we expect that the upper bounds  for the number of generators obtained by the partition shuffle spaces  are not optimal for all $l \ge 2$, i.e. we don't expect equality in Corollary \ref{ps-bounds}. We think that this reflects the existence of cusp forms as distinguished elements in depth $2$, whereas even zeta values just live in depth $1$. 
		By work of Ecalle we know that there is a Lie algebra structure on the partition shuffle spaces, see e.g. \cite{E} or \cite{schneps:arigari}. In forthcoming work we will study a sub Lie algebra which conjecturally has the algebra $\mathcal{A}$ as its symmetric algebra, which might give another explanation of this effect. Another optimistic hope is that a coproduct structure on $\zq$, which allows to mimic Brown's proof in order to obtain conjecture (B3), exists.
	\end{remark}

{\small
Henrik Bachmann, Nagoya University. \\
\texttt{henrik.bachmann@math.nagoya-u.ac.jp }\\

Ulf K\"uhn, Universit\"at Hamburg.\\
 \texttt{kuehn@math.uni-hamburg.de}\\	
}
	
\end{document}